\documentclass[12pt]{amsart}
\usepackage{amsmath,amssymb,amsthm,amsfonts,enumitem}
\usepackage{enumitem}
\usepackage[hmargin=30mm, vmargin=30mm]{geometry}
\usepackage[curve,tips]{xypic}
\SelectTips{eu}{11}
\UseTips
\usepackage{mathrsfs} 
\usepackage{xspace}
\usepackage[pagebackref=true, colorlinks, linkcolor={red},citecolor={blue}]{hyperref}

\newtheorem*{thm}{Main Theorem}

\newtheorem{prop}{Proposition}
\newtheorem{thm1}[prop]{Theorem}
\newtheorem{lem}[prop]{Lemma}

\theoremstyle{definition}

\newcommand{\Zb}{\mathbf{Z}}
\newcommand{\Nb}{\mathbf{N}}

\newcommand{\Wc}{\mathcal{W}}

\newcommand{\g}{\mathfrak{g}}

\newcommand{\h}{\mathfrak{h}}

\newcommand{\inv}{^{-1}}

\newcommand{\CC}{\mathrm{C}}

\newcommand{\la}{\langle}
\newcommand{\ra}{\rangle}

\DeclareMathOperator{\Ch}{Ch}
\DeclareMathOperator{\height}{ht}
\DeclareMathOperator{\dist}{d}

\newcommand{\co}{\colon\thinspace}

\usepackage[usenames, dvipsnames]{color}

\makeatletter 
\let\@wraptoccontribs\wraptoccontribs
\makeatother

\begin{document}
\title{Closed sets of real roots in Kac--Moody root systems}
\author[P.-E. Caprace]{Pierre-Emmanuel Caprace}
\thanks{P.-E.C. is a F.R.S.-FNRS senior research associate.} 
\address{Universit\'e catholique de Louvain, IRMP, Chemin du Cyclotron 2, bte L7.01.02, 1348 Louvain-la-Neuve, Belgique}
\email{pe.caprace@uclouvain.be}

\author[T. Marquis]{Timoth\'ee Marquis}
\thanks{T.M. is a F.R.S.-FNRS post-doctoral researcher.} 
\address{Universit\'e catholique de Louvain, IRMP, Chemin du Cyclotron 2, bte L7.01.02, 1348 Louvain-la-Neuve, Belgique}
\email{timothee.marquis@uclouvain.be}


\begin{abstract}
In this note, we provide a complete description of the closed sets of real roots in a Kac--Moody root system.
\end{abstract}

\maketitle


\section{Introduction}
Let $A=(a_{ij})_{i,j\in I}$ be a generalised Cartan matrix, and let $\g(A)$ be the associated Kac--Moody algebra (see \cite{Kac}). Like (finite-dimensional) semisimple complex Lie algebras, $\g(A)$ possesses a \emph{root space decomposition}
$$\g(A)=\h\oplus\bigoplus_{\alpha\in\Delta}\g_{\alpha}$$
with respect to the adjoint action of a \emph{Cartan subalgebra} $\h$, with associated set of \emph{roots} $\Delta\subseteq\h^*$ contained in the $\mathbf{Z}$-span of the set $\Pi=\{\alpha_i \ | \ i\in I\}$ of \emph{simple roots}, as well as a \emph{Weyl group} $\Wc\subseteq\mathrm{GL}(\h^*)$ stabilising $\Delta$. However, as soon as $A$ is not a Cartan matrix (i.e. as soon as $\g(A)$ is infinite-dimensional), the set $\Delta^{re}:=\Wc.\Pi$ of \emph{real roots} is properly contained in $\Delta$. In some sense, the real roots of $\Delta$ are those that behave as the roots of a semisimple Lie algebra; in particular, $\dim\g_{\alpha}=1$ for all $\alpha\in\Delta^{re}$.

A subset $\Psi\subseteq\Delta$ is \emph{closed} if $\alpha+\beta\in\Psi$ whenever $\alpha,\beta\in\Psi$ and $\alpha+\beta\in\Delta$. Note that, denoting by $\g_{\Psi}$ the subspace $\g_{\Psi}:=\bigoplus_{\alpha\in\Psi}\g_{\alpha}$ of $\g(A)$, a subset $\Psi\subseteq\Delta^{re}$ is closed if and only if $\h\oplus\g_{\Psi}$ is a subalgebra of $\g(A)$; in particular, if $\Psi\subseteq\Delta^{re}$ is closed, the subalgebra generated by $\g_{\Psi}$ is contained in $\h\oplus\g_{\Psi}$. Closed sets of real roots in Kac--Moody root systems thus arise naturally, and the purpose of this note is to provide a complete description of these sets. Our main theorem is as follows. For each $\alpha\in\Delta^{re}$, let $\alpha^{\vee}\in\h$ be the \emph{coroot} of $\alpha$, i.e. the unique element of $[\g_{\alpha},\g_{-\alpha}]$ with $\la\alpha,\alpha^{\vee}\ra=2$.

\begin{thm}
Let $\Psi\subseteq\Delta^{re}$ be a closed set of real roots and let $\g$ be the subalgebra of $\g(A)$ generated by $\g_{\Psi}$. Set $\Psi_s:=\{\alpha\in\Psi \ | \ -\alpha\in\Psi\}$ and $\Psi_n:=\Psi\setminus\Psi_s$. Set also $\h_s:=\sum_{\gamma\in\Psi_s}\CC\gamma^{\vee}$, $\g_s:=\h_s\oplus\g_{\Psi_s}$ and $\g_n:=\g_{\Psi_n}$. Then
\begin{enumerate}
\item
$\g_s$ is a subalgebra and $\g_n$ is an ideal of $\g$. In particular, $\g=\h_s\oplus\g_{\Psi}=\g_s\ltimes\g_n$.
\item
$\g_n$ is nilpotent; it is the largest nilpotent ideal of $\g$.
\item
$\g_s$ is a semisimple finite-dimensional Lie algebra with Cartan subalgebra $\h_s$ and set of roots $\Psi_s$.
\end{enumerate}
\end{thm}

Note that the possible closed root subsystems $\Psi_s$ in the statement of the Main Theorem were explicitely determined in \cite{KV18} when $A$ is an affine GCM.

\section{Preliminaries}

\subsection{Kac--Moody algebras}
The general reference for this section is \cite[Chapters~1--5]{Kac}. 

Let $A=(a_{ij})_{i,j\in I}$ be a {\bf generalised Cartan matrix} with indexing set $I$, and let $(\h,\Pi,\Pi^{\vee})$ be a realisation of $A$ in the sense of \cite[\S 1.1]{Kac}, with set of {\bf simple roots} $\Pi=\{\alpha_i \ | \ i\in I\}$ and set of {\bf simple coroots} $\Pi^{\vee}=\{\alpha_i^{\vee} \ | \ i\in I\}$. Let $\g(A)$ be the corresponding {\bf Kac--Moody algebra} (see \cite[\S 1.2-1.3]{Kac}). Then $\g(A)$ admits a root space decomposition 
$$\g(A)=\h\oplus\bigoplus_{\alpha\in\Delta}\g_{\alpha}$$
with respect to the adjoint action of the {\bf Cartan subalgebra} $\h$, with associated set of {\bf roots} $\Delta\subseteq\h^*$. Set $Q:=\bigoplus_{i\in I}\Zb\alpha_i\subseteq\h^*$ and $Q_+:=\bigoplus_{i\in I}\Nb\alpha_i\subseteq Q$. Then $\Delta\subseteq Q_+\cup -Q_+$, and we let $\Delta_+:=\Delta\cap Q_+$ (resp. $\Delta_-:=\Delta\setminus\Delta_+=-\Delta_+$) denote the set of {\bf positive} (resp. {\bf negative}) {\bf roots}. The {\bf height} of a root $\alpha=\sum_{i\in I}n_i\alpha_i\in Q$ is the integer $\height(\alpha):=\sum_{i\in I}n_i$. Thus a root is positive if and only if it has positive height. We also introduce a partial order $\leq$ on $Q$ defined by 
$$\alpha\leq\beta\iff \beta-\alpha\in Q_+.$$

The {\bf Weyl group} of $\g(A)$ is the subgroup $\Wc$ of $\mathrm{GL}(\h^*)$ generated by the {\bf fundamental reflections}
$$r_i=r_{\alpha_i}\co\h^*\to\h^*: \alpha\mapsto \alpha-\langle\alpha,\alpha_i^{\vee}\rangle\alpha_i$$
for $i\in I$. Alternatively, $\Wc$ can be identified with the subgroup of $\mathrm{GL}(\h)$ generated by the reflections $r_i\co\h\to\h:h\mapsto h-\la\alpha_i,h\ra\alpha_i^{\vee}$. Then $\Wc$ stabilises $\Delta$, and we let $\Delta^{re}:=\Wc.\Pi$ (resp. $\Delta^{re}_+:=\Delta^{re}\cap\Delta_+$) denote the set of (resp. {\bf positive}) {\bf real roots}. For each $\alpha\in\Delta^{re}$, say $\alpha=w\alpha_i$ for some $w\in\Wc$ and $i\in I$, the element $\alpha^{\vee}:=w\alpha_i^{\vee}$ only depends on $\alpha$, and is called the {\bf coroot} of $\alpha$. One can then also define the {\bf reflection} $r_{\alpha}\in\Wc$ associated to $\alpha$ as 
$$r_{\alpha}=wr_iw\inv\co \h^*\to\h^*: \beta\mapsto \beta-\la\beta,\alpha^{\vee}\ra\alpha.$$

Let $\Psi\subseteq\Delta$ be a subset of roots. We call $\Psi$ {\bf closed} if $\alpha+\beta\in\Psi$ whenever $\alpha,\beta\in\Psi$ and $\alpha+\beta\in\Delta$. The {\bf closure} $\overline{\Psi}$ of $\Psi$ is the smallest closed subset of $\Delta$ containing $\Psi$. A subset $\Psi'\subseteq\Psi$ is an {\bf ideal} in $\Psi$ if $\alpha+\beta\in\Psi'$ whenever $\alpha\in\Psi$, $\beta\in\Psi'$ and $\alpha+\beta\in\Delta$. The set $\Psi$ is {\bf prenilpotent} if there exist some $w,w'\in\Wc$ such that $w\Psi\subseteq \Delta_+$ and $w'\Psi\subseteq\Delta_-$; in that case, $\Psi$ is finite and contained in $\Delta^{re}$. If $\Psi$ is both prenilpotent and closed, it is called {\bf nilpotent}. We further call $\Psi$ {\bf pro-nilpotent} if it is a directed union of nilpotent subsets. 

The above terminology is motivated by its Lie algebra counterpart: consider the subspace $\g_{\Psi}:=\bigoplus_{\alpha\in\Psi}\g_{\alpha}$ for each $\Psi\subseteq\Delta$. If $\Psi$ is closed, then $\h\oplus\g_{\Psi}$ is a subalgebra. If, moreover, $\Psi\cap-\Psi=\varnothing$ and $\Psi'$ is an ideal in $\Psi$, then $\g_{\Psi'}$ is an ideal in $\g_{\Psi}$. If $\Psi$ is nilpotent, then $\g_{\Psi}$ is a nilpotent subalgebra. This remains valid for pro-nilpotent sets of roots.
\begin{lem}\label{lemma:pronilpotent_nilpotent}
Let $\Psi\subseteq\Delta$ be a pro-nilpotent set of roots. Then $\g_{\Psi}$ is a nilpotent subalgebra.
\end{lem}
\begin{proof}
By definition, $\g_{\Psi}$ is a directed union of nilpotent subalgebras $\g_{\Psi_n}$ associated to nilpotent sets of roots $\Psi_n\subseteq\Psi$. Since there is a uniform bound on the nilpotency class of these subalgebras by \cite[Theorem~1.1]{Cap_nil}, the claim follows.
\end{proof}

\subsection{Davis complexes}
The general reference for this section is \cite{BrownAbr} (see also \cite{Nos11}). 

The pair $(\Wc,S:=\{r_i \ | \ i\in I\})$ is a Coxeter system, and we let $\Sigma=\Sigma(\Wc,S)$ denote the corresponding {\bf Davis complex}. Thus $\Sigma$ is a $\mathrm{CAT}(0)$ cell complex whose underlying $1$-skeleton $\Sigma^{(1)}$ is the Cayley graph of $(\Wc,S)$. Moreover, the $\Wc$-action on $\Sigma$ is by cellular isometries, is proper, and induces on $\Sigma^{(1)}$ the canonical $\Wc$-action on its Cayley graph. The vertices of $\Sigma^{(1)}$ (which we identify with the elements of $\Wc$) are called the {\bf chambers} of $\Sigma$, and their set is denoted $\Ch(\Sigma)$. The vertex $1_{\Wc}$ is called the {\bf fundamental chamber} of $\Sigma$, and is denoted $C_0$. Two chambers of $\Sigma$ are {\bf adjacent} if they are adjacent in the Cayley graph of $(\Wc,S)$.  
   A {\bf gallery} is a sequence $\Gamma=(x_0,x_1,\dots,x_d)$ of chambers such that $x_{i-1}$ and $x_i$ are distinct adjacent chambers for each $i=1,\dots,d$. The integer $d\in\Nb$ is the {\bf length} of $\Gamma$, and $\Gamma$ is called {\bf minimal} if it is a gallery of minimal length between $x_0$ and $x_d$. In that case, $d$ is called the {\bf chamber distance} between $x_0$ and $x_d$, denoted $\mathrm{d}_{\Ch}(x_0,x_d)$. If $X$ is a nonempty subset of $\Ch(\Sigma)$ and $x\in\Ch(\Sigma)$, we also set $\mathrm{d}_{\Ch}(x,X):=\min_{x'\in X}\mathrm{d}_{\Ch}(x,x')$.

Given two distinct adjacent chambers $x,x'\in\Ch(\Sigma)$, any chamber of $\Sigma$ is either closer to $x$ or to $x'$ (for the chamber distance). This yields a partition of $\Ch(\Sigma)$ into two subsets, which are the underlying chamber sets of two closed convex subcomplexes of $\Sigma$, called {\bf half-spaces}. If $x=C_0$ and $x'=sC_0$ for some $s\in S$, we let $H(\alpha_s)$ denote the corresponding half-space containing $C_0$ and $H(-\alpha_s)$ the other half-space. If $x=wC_0$ and $x'=wsC_0$ ($w\in\Wc$ and $s\in S$), the corresponding half-spaces $H(\pm\alpha)=wH(\pm\alpha_s)$ depend only on $\alpha=w\alpha_s$. The map
$$\Delta^{re}\to \{\textrm{half-spaces of $\Sigma$}\}: \alpha\mapsto H(\alpha)$$
is a $\Wc$-equivariant bijection mapping $\Delta^{re}_+$ onto the set of half-spaces of $\Sigma$ containing $C_0$.
The {\bf wall} associated to $\alpha\in\Delta^{re}$ is the closed convex subset $\partial\alpha:=H(\alpha)\cap H(-\alpha)$ of $\Sigma$, and coincides with the fixed-point set of the reflection $r_{\alpha}\in\Wc$. A wall $\partial\alpha$ {\bf separates} two chambers $x,x'\in\Ch(\Sigma)$ if $x\in H(\epsilon\alpha)$ and $x'\in H(-\epsilon\alpha)$ for some $\epsilon\in\{\pm 1\}$. 

We conclude this preliminary section with a short dictionary between roots and half-spaces. Call two roots $\alpha,\beta\in\Delta^{re}$ {\bf nested} if either $H(\alpha)\subseteq H(\beta)$ or $H(\beta)\subseteq H(\alpha)$.
\begin{lem}\label{lemma:dictionary}
Let $\alpha,\beta\in\Delta^{re}$ with $\alpha\neq\pm\beta$, and set $m:=\la \alpha,\beta^{\vee}\ra$ and $n:=\la\beta,\alpha^{\vee}\ra$.
\begin{enumerate}
\item
$\{\alpha,\beta\}$ is not prenilpotent $\iff$ $\{\alpha,-\beta\}$ is nested $\iff$ $m,n<0$ and $mn\geq 4$ $\iff$ $\overline{\{\alpha,\beta\}}\not\subseteq\Delta^{re}$.
\item
$\la r_{\alpha},r_{\beta}\ra\subseteq\Wc$ is finite $\iff$ $\partial\alpha\cap\partial\beta\neq\varnothing$ $\iff$ $\{\pm\alpha,\pm\beta\}$ does not contain any nested pair $\iff$ $\overline{\{\pm\alpha,\pm\beta\}}$ is a root system of type $A_1\times A_1$, $A_2$, $B_2$ or $G_2$.
\item
A subset $\Psi\subseteq\Delta^{re}$ of roots is prenilpotent if and only if the subcomplexes $\bigcap_{\alpha\in\Psi}H(\alpha)$ and $\bigcap_{\alpha\in\Psi}H(-\alpha)$ of $\Sigma$ both contain at least one chamber.
\end{enumerate}
\end{lem}
\begin{proof}
(1) The first equivalence is \cite[Lemma~8.42(3)]{BrownAbr}, and the second and third equivalences follow from \cite[Proposition~4.31(3)]{thesemoi} (see also \cite[Exercise~7.42]{KMGbook}).

(2) For the first equivalence, note that $\la r_{\alpha},r_{\beta}\ra$ fixes $\partial\alpha\cap\partial\beta$. Hence if $\partial\alpha\cap\partial\beta\neq\varnothing$, then $\la r_{\alpha},r_{\beta}\ra$ is finite (because the $\Wc$-action on $\Sigma$ is proper, hence has finite point stabilisers). Conversely, if $\partial\alpha\cap\partial\beta=\varnothing$, then $r_{\alpha},r_{\beta}$ generate an infinite dihedral group.

For the second equivalence, note that since half-spaces are convex, they are in particular arc-connected, and hence if $\partial\alpha\cap\partial\beta=\varnothing$ then either $\{\alpha,\beta\}$ or $\{\alpha,-\beta\}$ is nested. Conversely, if $\{\pm\alpha,\pm\beta\}$ contains a nested pair, then $\la r_{\alpha},r_{\beta}\ra\subseteq\Wc$ is an infinite dihedral group.

Finally, for the third equivalence, note that $\la r_{\alpha},r_{\beta}\ra$ stabilises $R:=\overline{\{\pm\alpha,\pm\beta\}}$, and that $\la r_{\alpha},r_{\beta}\ra$ is finite if and only if $R$ is finite. In this case, $R$ is a (reduced) root system of rank $2$ in the sense of \cite[VI, \S 1.1]{Bourbaki}, and hence of one of the types $A_1\times A_1$, $A_2$, $B_2$ or $G_2$ by \cite[VI, \S 4.2 Th\'eor\`eme 3]{Bourbaki}, as desired.

(3) $\Psi$ is prenilpotent if and only if there exist $v,w\in\Wc$ such that $v\Psi\subseteq \Delta^{re}_+$ (i.e. $v\inv C_0\in\bigcap_{\alpha\in\Psi}H(\alpha)$) and $w\Psi\subseteq \Delta^{re}_-$ (i.e. $w\inv C_0\in\bigcap_{\alpha\in\Psi}H(-\alpha)$).
\end{proof}

\section{Closed and prenilpotent sets of roots}

We shall need the following lemma, which is a slight variation of  \cite[Lemma~12]{CapAGT06}.  

\begin{lem}\label{lem:12}
Let $x\in \Ch(\Sigma)$ be a chamber and $\alpha \in \Delta^{re}$ be a root such that $x \not \in H(\alpha)$. Let $y \in \Ch(\Sigma)$ be a chamber contained in $H(\alpha)$ and at minimal distance from $x$, and let $\beta \in \Delta^{re}$ be a root such that $H(\beta)$ contains $x$ but not $y$, and that $H(-\beta)$ contains a chamber adjacent to $x$. Assume that $\beta \neq -\alpha$. Then:
\begin{enumerate}[label=(\roman*)]
\item $\langle \alpha, \beta^\vee \rangle < 0$. 
\item $ x \not \in H(r_\beta(\alpha))$. 
\end{enumerate}	
\end{lem}	
\begin{proof}
For the first assertion, we invoke \cite[Lemma~12]{CapAGT06} that we apply to the roots $\phi:=\alpha$ and $\psi:=\beta$. We deduce  that   $r_\alpha(\beta) \neq \beta$ and that $H(r_{\alpha}(\beta))\supseteq H(\alpha)\cap H(\beta)$. If $\langle r_\alpha, r_\beta\rangle$ is infinite, then $\partial\alpha\cap\partial\beta=\varnothing$, and since $\{\alpha,\beta\}$ is not nested by construction, the pair $\{\alpha,-\beta\}$ is nested and $\la\alpha,\beta^{\vee}\ra<0$ (see Lemma~\ref{lemma:dictionary}(1)). If $\langle r_\alpha, r_\beta\rangle$ is finite, on the other hand, the root system generated by $\alpha$ and $\beta$ is an irreducible root system of rank~$2$ (see  Lemma~\ref{lemma:dictionary}(2)). A quick case-by-case inspection  of the rank 2 root systems of type $A_2$, $B_2$ and $G_2$ (see \cite[\S 9.3]{Hum78}) reveals that the condition $H(r_{\alpha}(\beta))\supseteq H(\alpha)\cap H(\beta)$ implies that the angle between the walls of $\alpha$ and $\beta$ (in the Euclidean plane spanned by $\alpha,\beta$ and with scalar product $(\cdot|\cdot)$ such that $\la\alpha, \beta^\vee\ra=2(\alpha|\beta)/(\alpha|\alpha)$, see \cite[\S 9.1]{Hum78}) is acute. Thus the angle between the corresponding roots is obtuse, hence $\la\alpha, \beta^\vee\ra<0$. This proves (i).

For the second assertion, let $x' \in \Ch(\Sigma)$ be the unique chamber adjacent to $x$ and contained in   $H(-\beta)$. If $x \in H(r_\beta(\alpha))$, then $x' = r_\beta(x) \in r_\beta(H(r_\beta(\alpha))) = H(\alpha)$. Since $x \not \in H(\alpha)$ by hypothesis, it follows that the wall $\partial \alpha$ separates $x$ from $x'$. Since $x$ and $x'$ are adjacent, we deduce that $\beta = -\alpha$, contradicting the hypotheses.  This proves (ii). 
\end{proof}

Given a set of roots $\Phi\subseteq\Delta$ and a subset $\Psi \subseteq \Phi$, we set 
$$\Phi_{\geq \Psi} := \{\phi \in \Phi \mid \phi \geq \psi \text{ for some } \psi \in \Psi\}.$$

\begin{lem}\label{lem:NonEmptyIntersec}
Let $\Phi\subseteq\Delta^{re}$ be a  closed set of roots and let $\Psi \subseteq \Phi$ be a non-empty finite subset. Assume that $\Phi \cap - \Phi = \varnothing$. Then $\bigcap_{\phi\in\Phi_{\geq \Psi}} H(\phi)$ contains a chamber. In particular, $\bigcap_{\psi\in\Psi} H(\psi)$ contains a chamber.
\end{lem}
\begin{proof}
For simplicity, we identify each half-space $H(\phi)$ ($\phi\in\Delta^{re}$) with its underlying set of chambers; we thus have to show that $\bigcap_{\phi\in\Phi_{\geq \Psi}} H(\phi)\neq\varnothing$.
For each set of roots $B \subseteq \Delta$, we set $B_\varepsilon = B \cap \Delta_\varepsilon$ for $\varepsilon \in \{+, -\}$. 

Set $A = \Phi_{\geq \Psi}$. Observe that the set $A_-$ is finite, since $\Psi$ is finite by hypothesis. We shall proceed by   induction on $|A_-|$. 

In the base case $|A_-| = 0$, we have $A \subseteq \Delta^{re}_+$, and each half-space $H(\phi)$ with $\phi \in A$ contains the fundamental chamber $C_0$. Thus $\bigcap_{\phi\in A} H(\phi)$ is nonempty in this case.

We assume henceforth that $A_-$ is non-empty. We set $\mathcal X = \bigcap_{\phi \in \Phi_+} H(\phi)$. In the special case where $\Phi_+$ is empty, we adopt the convention that $\mathcal X = \Ch(\Sigma)$. In all cases, we observe that $\mathcal X$ is non-empty since it contains  $C_0$.  
We now distinguish two cases. 

Assume first that there exists $\psi \in A_-$ such that $\mathcal X$ is not entirely contained in the half-space $H(-\psi)$. We may then choose a chamber $C \in \mathcal X \cap H(\psi)$. Let $w \in W$ be such that $wC = C_0$. For each $\phi \in \Phi_+$, we have $C \in \mathcal X \subseteq H(\phi)$, hence $C_0 \in H(w\phi)$. Therefore  $w\Phi_+ \subseteq \Delta_+$, and in particular $wA_+ \subseteq \Delta_+$. Moreover $wA \cap \Delta_+$ contains $w\psi$, since $C \in H(\psi)$. It follows that $|wA \cap \Delta_-| = |wA_- \cap \Delta_-| < |wA_-|= |A_-|$. Set 
$$B = wA  = w(\Phi_{\geq \Psi}) = w(\Phi)_{\geq w(\Psi)}.$$
We have just seen that $|B_-| < |A_-|$. Hence the induction hypothesis ensures that $\bigcap_{\phi\in B} H(\phi)\neq\varnothing$. Therefore, $\bigcap_{\phi\in A} H(\phi) = w^{-1}\bigg( \bigcap_{\phi\in B} H(\phi)\bigg)$ is also non-empty, and we are done in this case.

Assume finally that for all $\psi \in A_-$, we have  $\mathcal X \subseteq H(-\psi)$. Choose $x \in \mathcal X$ and $\alpha \in A_-$ such that $\dist_{\Ch}(x, H(\alpha))$ is minimal. Choose also $y \in H(\alpha)$ such that $\dist_{\Ch}(x, y) = \dist_{\Ch}(x, H(\alpha))$. Let $x = x_0, x_1, \dots, x_m = y$ be a minimal gallery from $x$ to $y$, and let $\beta \in \Delta^{re}$ be the unique root such that $H(\beta)$ contains $x_0$ but not $x_1$. 

We have $\beta \in \Phi_+$, since otherwise we would have    $x_1 \in H(\phi)$ for all $\phi \in \Phi_+$, so that $x_1 \in \mathcal X$, contradicting the minimality condition in the definition of $x$.  In particular, we have $\alpha \neq - \beta$ since  $\Phi \cap - \Phi = \varnothing$. 

Notice that $x$ is a chamber contained in $H(\beta)$ but not in $H(\alpha)$, whereas $y$ is a chamber contained in $H(\alpha)$ but not $H(\beta)$. We now invoke Lemma~\ref{lem:12}. It follows that $\langle \alpha, \beta^\vee \rangle < 0$ and that $ x \not \in H(r_\beta(\alpha))$. Since $\beta$ is a positive root, it follows that $r_\beta(\alpha) = \alpha - \langle \alpha, \beta^\vee \rangle \beta \geq \alpha \in A = \Phi_{\geq \Psi}$. Moreover $r_\beta(\alpha) \in \overline{\{\alpha, \beta\}} \subseteq \Phi$. Hence $r_\beta(\alpha) \in A$. In particular, $r_\beta(\alpha)\in\Delta_-$, for otherwise $x\in\mathcal X\subseteq H(r_\beta(\alpha))$, a contradiction. We conclude that $r_\beta(\alpha) \in A_-$. Now we observe that the gallery $x = x_0 = r_\beta(x_1), r_\beta(x_2), \dots, r_\beta(x_m) = r_\beta(y)$ is of length~$m-1$ and joins $x$ to a chamber in $H(r_{\beta}(\alpha))$. Since  $r_\beta(\alpha)$ belongs to $A_-$, this contradicts the minimality condition in the definition of $\alpha$. Thus this final case does not occur, and the proof is complete.
\end{proof}

\section{Nilpotent sets}

\begin{lem}\label{lemma:nilpotency_caract}
A subset $\Phi\subseteq\Delta^{re}$ is nilpotent if and only if it satisfies the following three conditions:
\begin{enumerate}[label=(\roman*)]
\item $\Phi$ is closed.
\item $\Phi$ is finite.
\item $\Phi \cap -\Phi = \varnothing$. 
\end{enumerate}
\end{lem}
\begin{proof}
Assume that $\Phi\subseteq\Delta^{re}$ satisfies the three conditions (i)--(iii). Applying Lemma~\ref{lem:NonEmptyIntersec} with $\Phi = \Psi$, we deduce that $\bigcap_{\phi\in\Phi} H(\phi)$ contains a chamber. Similarly, Lemma~\ref{lem:NonEmptyIntersec} applied to $-\Phi$ implies that $\bigcap_{\phi\in\Phi} H(-\phi)$ also contains a chamber. Hence Lemma~\ref{lemma:dictionary}(3) implies that $\Phi$ is nilpotent, as desired. 

The converse assertion is clear by the definition of a nilpotent set of roots.
\end{proof}

\section{On the closure of a finite set}

\begin{lem}\label{lemma:ClosureFinite}
Let $\Phi\subseteq\Delta^{re}$ be a closed set of roots such that 
$\Phi \cap -\Phi = \varnothing$. For any  finite subset $\Psi \subseteq \Phi$, the closure $\overline \Psi$ is finite.
\end{lem}
\begin{proof}
Applying Lemma~\ref{lem:NonEmptyIntersec} to $\Psi$ and $-\Psi$, we deduce that the intersections 
$\bigcap_{\psi\in\Psi} H(\psi)$ and $\bigcap_{\psi\in\Psi} H(-\psi)$ both contain a chamber. Hence $\Psi$ is prenilpotent by  Lemma~\ref{lemma:dictionary}(3). The closure of any prenilpotent set of real roots is nilpotent, hence finite by Lemma~\ref{lemma:nilpotency_caract}. 
\end{proof}

\section{Pro-nilpotent sets}

\begin{prop}\label{prop:pronilpotent}
	Let $\Phi\subseteq\Delta^{re}$ be a closed set of real roots such that $\Phi\cap-\Phi=\varnothing$. Then $\Phi$ is pro-nilpotent.
\end{prop}
\begin{proof}
Let $n$ be a positive integer and let $\Phi_{\leq n}$ denote the subset of those $\phi \in \Phi$ with $|\height(\phi)|\leq n$. Thus 	$\Phi_{\leq n}$  is finite for all $n$, and the sets $\Phi_{\leq n}$ are linearly ordered by inclusion. By Lemma~\ref{lemma:ClosureFinite}, the closure $ \overline{\Phi_{\leq n}}$ is finite, hence nilpotent by  Lemma~\ref{lemma:nilpotency_caract}. Thus $\Phi$ is the union of an ascending chain of nilpotent subsets.  
\end{proof}

\section{Levi decomposition}

\begin{lem}\label{lemma:Levi_dec_roots}
Let $\Psi\subseteq\Delta^{re}$ be a closed set of real roots. Set $\Psi_s:=\{\alpha\in\Psi \ | \ -\alpha\in\Psi\}$ and $\Psi_n:=\Psi\setminus\Psi_s$. Then $\Psi_s$ is closed and $\Psi_n$ is an ideal in $\Psi$. 
\end{lem}
\begin{proof}
If $\alpha,\beta\in\Psi_s$ and $\alpha+\beta\in\Delta$, then $-\alpha,-\beta\in\Psi_s$ and hence $-(\alpha+\beta)\in\Psi$, that is, $\alpha+\beta\in\Psi_s$. This shows that $\Psi_s$ is closed. 

If now $\alpha\in\Psi$ and $\beta\in\Psi_n$ are such that $\alpha+\beta\in\Delta$, then $\alpha+\beta\in\Psi_n$. Otherwise, $-(\alpha+\beta)\in\Psi$, and hence $-\beta=-(\alpha+\beta)+\alpha\in\Psi$, contradicting the fact that $\beta\in\Psi_n$. This shows that $\Psi_n$ is an ideal in $\Psi$.
\end{proof}

\begin{lem}\label{lemma:finite_root_system}
Let $\Psi\subseteq\Delta^{re}$ be a closed set of real roots such that $\Psi=-\Psi$. Then $\Psi$ is the root system  of a finite-dimensional semisimple complex Lie algebra
\end{lem}
\begin{proof}
Let $H$ be the reflection subgroup of $\mathcal W$ generated by $R_{\Psi}:=\{r_{\alpha} \ | \ \alpha\in\Psi\}$. Then $H$ is a Coxeter group in its own right (see \cite{Deo89}). Moreover, the set $M_{\Psi}$ of walls of $\Sigma(\Wc,S)$ corresponding to reflections in $R_{\Psi}$ is stabilised by $H$. Indeed, if $\alpha,\beta\in\Psi$, we have to check that $r_{\alpha}(\beta)=\beta-\beta(\alpha^{\vee})\alpha\in\Psi$. But as $\pm\alpha\in\Psi$, this follows from the fact that $\Psi$ is closed.
In particular, $H$ is finite, for otherwise it would contain two reflections $r_{\alpha},r_{\beta}\in R_{\Psi}$ generating an infinite dihedral group by \cite[Prop. 8.1, p. 309]{Hee}, and hence $\Psi$ would contain a non-prenilpotent pair of roots by Lemma~\ref{lemma:dictionary}(1,2), contradicting the fact that $\Psi\subseteq\Delta^{re}$ (see Lemma~\ref{lemma:dictionary}(1)). This shows that $\Psi$ is a finite (reduced) root system in the sense of \cite[VI, \S 1.1]{Bourbaki}, hence the root system of a finite-dimensional semisimple complex Lie algebra (see \cite[VI, \S 4.2 Th\'eor\`eme 3]{Bourbaki}), as desired.
\end{proof}

\begin{thm1}
Let $\Psi\subseteq\Delta^{re}$ be a closed set of real roots and let $\g$ be the subalgebra of $\g(A)$ generated by $\g_{\Psi}$. Set $\Psi_s:=\{\alpha\in\Psi \ | \ -\alpha\in\Psi\}$ and $\Psi_n:=\Psi\setminus\Psi_s$. Set also $\h_s:=\sum_{\gamma\in\Psi_s}\CC\gamma^{\vee}$, $\g_s:=\h_s\oplus\g_{\Psi_s}$ and $\g_n:=\g_{\Psi_n}$. Then
\begin{enumerate}
\item
$\g_s$ is a subalgebra and $\g_n$ is an ideal of $\g$. In particular, $\g=\g_s\ltimes\g_n$.
\item
$\g_n$ is the unique maximal nilpotent ideal of $\g$.
\item
$\g_s$ is a semisimple finite-dimensional Lie algebra with Cartan subalgebra $\h_s$ and set of roots $\Psi_s$.
\end{enumerate}
\end{thm1}
\begin{proof}
(1) follows from Lemma~\ref{lemma:Levi_dec_roots} and (3) from Lemma~\ref{lemma:finite_root_system}. For (2), note that $\Psi_n$ is a pro-nilpotent set of roots by Proposition~\ref{prop:pronilpotent}, and hence $\g_n$ is nilpotent by Lemma~\ref{lemma:pronilpotent_nilpotent}. Finally, since $\g_s$ is semisimple, the image of any nilpotent ideal $\mathfrak i$ of $\g$ under the quotient map $\g\to\g_s$ must be zero, that is, $\mathfrak i\subseteq\g_n$. Therefore, $\g_n$ is the unique maximal nilpotent ideal of $\g$, as desired.
\end{proof}


\bibliographystyle{amsalpha}
\bibliography{biblio-nil}

\end{document}